\pdfoutput=1
\RequirePackage{ifpdf}
\ifpdf 
\documentclass[pdftex]{sigma}
\else
\documentclass{sigma}
\fi

\usepackage{euscript}
\usepackage{amscd}
\usepackage{yhmath}
\usepackage{mathrsfs}
\usepackage{bbding}

\numberwithin{equation}{section}
\usepackage[all]{xy}

\newtheorem{thm}{Theorem}[section]
\newtheorem{cor}[thm]{Corollary}
\newtheorem{lem}[thm]{Lemma}

\newtheorem{prop}[thm]{Proposition}

\newtheorem{clm}[thm]{Claim}

\theoremstyle{definition}
\newtheorem{defn}[thm]{Definition}
\newtheorem{rem}[thm]{Remark}

 \newcommand{\ba}{\begin{eqnarray}}
  \newcommand{\na}{\end{eqnarray}}

\def\E{\ifmmode{\mathbb E}\else{$\mathbb E$}\fi} %
\def\N{\ifmmode{\mathbb N}\else{$\mathbb N$}\fi} 
\def\R{\ifmmode{\mathbb R}\else{$\mathbb R$}\fi} 
\def\Q{\ifmmode{\mathbb Q}\else{$\mathbb Q$}\fi} 
\def\C{\ifmmode{\mathbb C}\else{$\mathbb C$}\fi} 
\def\H{\ifmmode{\mathbb H}\else{$\mathbb H$}\fi} 
\def\Z{\ifmmode{\mathbb Z}\else{$\mathbb Z$}\fi} 
\def\P{\ifmmode{\mathbb P}\else{$\mathbb P$}\fi} %
\def\T{\ifmmode{\mathbb T}\else{$\mathbb T$}\fi} %
\def\SS{\ifmmode{\mathbb S}\else{$\mathbb S$}\fi} %
\def\DD{\ifmmode{\mathbb D}\else{$\mathbb D$}\fi} %
\def\K{\ifmmode{\mathbb K}\else{$\mathbb K$}\fi}
\def\F{\ifmmode{\mathbb F}\else{$\mathbb F$}\fi} 

\newcommand{\arc}[1]{%
 \settowidth{\dimen0}{\ensuremath{#1}}%
 \divide\dimen0 by 2%
 \overset{\rotatebox{-90}{\ensuremath{\left(\rule{0pt}{\dimen0}\right.}}}{#1}%
}

\begin{document}

\allowdisplaybreaks

\renewcommand{\thefootnote}{}

\renewcommand{\PaperNumber}{030}

\FirstPageHeading

\ShortArticleName{Sign Convention for $A_{\infty}$-Operations in Bott--Morse Case}

\ArticleName{Sign Convention for $\boldsymbol{A_{\infty}}$-Operations\\ in Bott--Morse Case\footnote{This paper is a~contribution to the Special Issue on Integrability, Geometry, Moduli in honor of Motohico Mulase for his 65th birthday. The~full collection is available at \href{https://www.emis.de/journals/SIGMA/Mulase.html}{https://www.emis.de/journals/SIGMA/Mulase.html}}}

\Author{Kaoru ONO}

\AuthorNameForHeading{K.~Ono}

\Address{Research Institute for Mathematical Sciences, Kyoto University, Kyoto, 606-8502, Japan}
\Email{\href{mailto:email@address}{ono@kurims.kyoto-u.ac.jp}}

\ArticleDates{Received July 14, 2024, in final form March 03, 2025; Published online April 29, 2025}

\Abstract{We describe the sign and orientation issue appearing the filtered $A_{\infty}$-formulae in Lagrangian Floer theory using de Rham model in Bott--Morse setting. After giving the definition of filtered $A_{\infty}$-operations in a Fukaya category, we verify the filtered $A_{\infty}$-formulae.}


\Classification{53D40; 53D37; 58A99}

\renewcommand{\thefootnote}{\arabic{footnote}}
\setcounter{footnote}{0}

\section{Introduction}
The aim of this note is to describe the sign and orientation issue appearing the filtered $A_{\infty}$-formulae in Lagrangian Floer theory
using de Rham model in Bott--Morse setting.
When we work with only one relatively spin Lagrangian submanifold,
we constructed the filtered $A_{\infty}$-algebra in~\cite{FOOO09I,FOOO09II} using the singular chain complex model.
The sign and orientation are explained in~\cite[Sections~8.3--8.5]{FOOO09II}.
In the de Rham model version, see~\cite[Section~22.4]{FOOO20} and also~\cite{ST}.
We gave a construction of the filtered $A_{\infty}$-bimodule using the singular chain model in~\mbox{\cite{FOOO09I,FOOO09II}}, especially,
the sign and orientation are described in~\cite[Section~8.8]{FOOO09II}.
Sign and orientation in Bott--Morse Hamiltonian Floer complex using the de Rham model version, see~\cite[Definition~19.3 and Proposition~19.5]{FOOO20}.
In this note, we discuss the sign and orientation issue appearing in the construction of the filtered $A_{\infty}$-category for a collection of finitely many
(relatively) spin Lagrangian submanifolds.
The construction of Kuranishi structures (a version of a tree-like K-system in the sense of~\cite{FOOO20}) on moduli spaces
of stable holomorphic polygons is discussed
in other papers~\cite{AFOOO,F}.
Here, we give a definition of $A_{\infty}$-operations in Bott--Morse case (see Definition~\ref{def}) using such Kuranishi structures.
We verify the sign convention by showing the~filtered $A_{\infty}$-relation (see Theorem~\ref{conclusion}).

\section{Preliminaries}
We use the convention on orientation on the fiber product (in the sense of Kuranishi structure) as in~\cite[Section 8.2]{FOOO09II}.
Let $p\colon M \to N$ be a fiber bundle with oriented relative tangent bundle.
Restrict the fiber bundle to an open subset $U$ of $N$, we may assume that $U$ is oriented.
Then we give an orientation on $p^{-1}(U) \subset M$ using the isomorphism
$TM = p^* TN \oplus T_{\rm fiber} M$, where $T_{\rm fiber}M$ is the relative tangent bundle.
Then our convention of the integration along fibers of~${p\colon M \to N}$ is
\begin{equation*} 
\int_U \alpha \wedge p_! \beta = \int_{p^{-1}(U)} p^* \alpha \wedge \beta, \nonumber
\end{equation*}
where $\alpha \in \Omega^*(U)$ and $\beta \in \Omega^*\bigl(p^{-1}(U)\bigr)$,
Reversing the orientation of $U$ induces reversing of the orientation of $p^{-1}(U)$, hence the push-forward $p_! \beta$ does not depend on
the choice of the orientation of $U$.
Therefore, for a proper submersion $p\colon M \to N$ with the oriented relative tangent bundle, the integration along fibers
\[
p_!\colon \ \Omega^k(M) \to \Omega^{k-\dim M + \dim N}(N)
\]
is well defined.

We have the following properties.

\begin{prop}\quad
\begin{enumerate}\itemsep=0pt
\item[$(1)$] $p_! ( ( p^* \theta ) \wedge \beta ) = \theta \wedge (p_! \beta),$ where $\theta \in \Omega^*(N)$ and
$\beta \in \Omega^*(M)$.

\item[$(2)$]
Let $p\colon M \to N$ and $q\colon N \to B$ be fiber bundles with oriented relative tangent bundles. For $\beta \in \Omega^*(M)$, we have
$(q \circ p)_! \beta = q_! \circ p_! (\beta)$.
\end{enumerate}
\end{prop}

Using them, we find the following.

\begin{cor}\label{iteratefib}
$(q \circ p)_! (p^* \theta \wedge \beta) = q_! (\theta \wedge p_! \beta)$.
\end{cor}


\begin{prop}[base change]\label{base change}
Let $f\colon S \to N$ be a smooth map.
Denote by $\overline{p}\colon f^*M \to S$ the pullback of the fiber bundle $p\colon M \to N$ and $\widetilde{f}\colon f^*M \to M$ the bundle map
covering $f$. Then we~have
$f^* \circ p_! = \overline{p}_! \circ \widetilde{f}^*$.
\end{prop}

\begin{prop}[Stokes type formula, {\cite[Theorem 9.28]{FOOO20}}]\label{Stokes-Pr}
Let $p\colon M \to N$ be a smooth map $($or a strongly smooth map from a space with Kuranishi structure to a smooth manifold$)$
\[{\rm d} p_! \beta = p_!  {\rm d}\beta +(-1)^{\dim M + \deg \beta} p\vert_{\partial M} \beta.\]
\end{prop}

We introduced the notions of a strongly smooth map and a weakly submersive strongly smooth map from a space
equipped with Kuranishi structure to a smooth manifold
in~\cite[Definition~3.40\,(4),~(5)]{FOOO20}. We call a space equipped with a Kuranishi structure a K-space for short.
For a proper weakly submersive strongly smooth map $p$ from a K-space $X$ to a manifold~$N$, we define the integration along fibers
using a CF-perturbation, see~\cite[Section~9.2]{FOOO20}.
In this note, we suppress the notation for Kuranishi structures or good coordinate systems as well as
CF perturbations. Refer the indicated places in~\cite{FOOO20} for detailed statements.
For the verification of the sign convention in the filtered $A_{\infty}$-relations, it is sufficient to treat the integration along fibers of
a proper weakly submersive strongly smooth map as if the one for proper submersion between smooth manifolds.

The statements above holds for a proper weakly submersive strongly smooth map $p$.
For Proposition \ref{base change}, $f^*M$ is the fiber product of $f\colon S \to N$ and $p\colon M \to N$.
When $S$ and $M$ are K-spaces with a strongly smooth map $f\colon S \to N$ in the sense of~\cite[Definition 3.40\,(4)]{FOOO20}
and a~weakly submersive strongly smooth map $p\colon M \to N$,
we have a compatible system of smooth maps from Kuranishi charts of the fiber product $S \times_N M$ to the manifold $N$ and
the obstruction bundle on a fiber product Kuranishi chart of $f^*M$ contains the pullback of the obstruction bundle on a~Kuranishi chart of $M$
as a subbundle.
Using the pullback CF perturbation on~$f^*M$, we obtain Proposition \ref{base change} in such a situation.

The integration along fibers changes the degree of differential forms by
\begin{equation}\label{deg pushout}
 \deg p_! \beta = \deg \beta - {\rm reldim}\ p.
 \end{equation}
Here ${\rm reldim} \ p = \dim X - \dim N$, where $\dim X$ is the dimension of $X$ in the sense of K-space, see~\cite[p.~52]{FOOO20}.
A tuple $(X, f_1\colon X \to M_1, f_2\colon X \to M_2)$ is called a smooth correspondence, if~$X$ is a K-space, $f_1$, $f_2$ are strongly smooth maps
and $f_1$ is weakly submersive.  After taking CF-perturbations, we define
\[{\rm Corr}_X\colon \ \Omega^*(M_2) \to \Omega^*(M_1)\]
by $(f_1)_! \circ (f_2)^*$.
For flat vector bundles ${\mathcal L}_i$ on $M_i$, $i = 1,2$, with a given isomorphism
\[f_1^* {\mathcal L}_1 \cong O_{f_1} \otimes f_2^* {\mathcal L}_2,\] where $O_{f_1}$ is the orientation bundle of the relative tangent bundle
of $f_1\colon X \to M_1$, Theorem~27.1 in~\cite{FOOO09II} gives
\[{\rm Corr}_X\colon \ \Omega^*(M_2, {\mathcal L}_2) \to \Omega^*(M_1, {\mathcal L}_1).\]

Using Proposition \ref{Stokes-Pr}, we have the following.

\begin{prop}[{\cite[Proposition 27.2]{FOOO20}}]\label{corrStokes}
\[  {\rm d} \circ {\rm Corr}_X \xi = {\rm Corr}_X \circ d \xi + (-1)^{\dim X + \deg \xi} {\rm Corr}_{\partial X} \xi \qquad {\rm for} \ \xi \in \Omega^*(M_2; {\mathcal L}_2).
\]
\end{prop}

Let $(X_{12}, f_{1,12}\colon X_{12} \to M_1, f_{2,12}\colon X_{12} \to M_2)$ and
$(X_{23}, f_{2,23}\colon X_{23} \to M_2, f_{3,23}\colon X_{23} \to M_3)$ be smooth correspondences with given isomorphisms
\begin{equation}\label{isomlocsys}
f_{1,12}^* {\mathcal L}_1 \cong O_{f_{1,12}} \otimes f_{2,12}^* {\mathcal L}_2, \qquad
f_{2,23}^* {\mathcal L}_2 \cong O_{f_{2,23}} \otimes f_{3,23}^* {\mathcal L}_3.
\end{equation}
Taking the fiber product $X_{13}$ over $f_{2,12}$ and $f_{2,23}$,
we obtain a smooth correspondence
\[
(X_{13}, f_{1,13}\colon X_{13} \to M_1, f_{3,13}\colon X_{13} \to M_3)
\]
with the isomorphism
\[f_{1,13}^* {\mathcal L}_1 \cong O_{f_{1,13}} \otimes f_{3,13}^* {\mathcal L}_3\]
induced by~\eqref{isomlocsys} and
\[O_{f_{1,13}} \cong g_1^* O_{f_{1,12}} \otimes g_2^* O_{f_{2,23}}.\]
Here we denote by $g_1\colon X_{13} \to X_{12}$ and $g_2\colon X_{13} \to X_{23}$ the projections of the fiber product of Kuranishi charts,
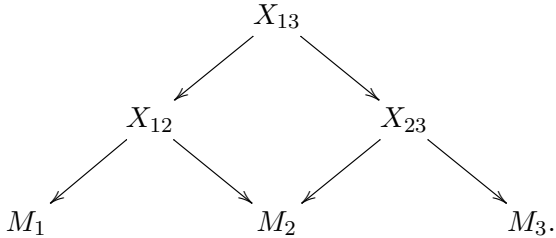
\begin{equation*}
\xymatrix{
&& X_{13} \ar[ld]\ar[rd] \\
& X_{12} \ar[ld]\ar[rd] && X_{23}\ar[ld]\ar[rd] \\
M_1 && M_2 && M_3.}
\end{equation*}

Then we have the following.
\begin{prop}[composition formula, {\cite[Theorem 10.21]{FOOO20}}]
\[{\rm Corr}_{X_{13}} = {\rm Corr}_{X_{12}} \circ {\rm Corr}_{X_{23}}.\]
\end{prop}

See~\cite[Chapter 27]{FOOO20} in the case with coefficients in local systems, see~\cite[Theorems~27.1 and~27.2]{FOOO20}.
In fact, the composition formula is a consequence of the properties mentioned above.

\section[Definition of A\_{infty}-operations]{Definition of $\boldsymbol{A_{\infty}}$-operations}

Let $\{L_i\}$ be a relatively spin collection of Lagrangian submanifolds, which intersects cleanly in~$(X, \omega)$.
In a later argument, we glue the linearization operator of holomorphic polygons with a Cauchy--Riemann type operator
at each boundary marked point, which is sent to the clean intersection of
two branches of relatively spin Lagrangian submanifolds, to obtain a Cauchy--Riemann type operator on the unit disk.
For the orientation issue, the argument works for clean intersections of distinct relative spin pair of Lagrangian submanifolds
and clean self-intersection of a relative spin Lagrangian submanifold.
The description of the boundary of holomorphic polygons in Lagrangian immersion case is found in the paper by Akaho and Joyce~\cite{AJ}.
For the sign and orientation issue, the argument presented here is also valid for immersed Lagrangian submanifolds.
Denote by $R_{\alpha}$ a connected component of $L_i$ and $L_j$. (We also consider the case of self clean intersection.)

Let $(\Sigma, \partial \Sigma)$ be a bordered Riemann surface $\Sigma$ of genus $0$ and with connected boundary
and $\vec{z}=(z_0, \dots, z_k)$ boundary marked points respecting the cyclic order on $\partial \Sigma$.
Let $u\colon (\Sigma, \partial \Sigma) \to (X, \cup L_i)$ be a smooth map such that \smash{$u\bigl(\arc{z_j z_{j+1}}\bigr) \subset L_{i_j}$}, $j \mod k+1$, $u(z_j) \in R_{\alpha_j}$, where
$R_{\alpha_j}$ is a~connected component of $L_{i_{j-1}} \cap L_{i_j}$.
(For an immersed Lagrangian with clean self intersection, $R_{\alpha}$ is a connected component of the clean intersection.)
For such $u$ and $u'$, we introduce the equivalence relation $\sim$ so that $u \sim u'$ when $\int_{\Sigma} \omega = \int_{\Sigma'} \omega$ and the Maslov indices of $u$ and $u'$
are the same. Denote by $B$ the equivalence class.
In this note, the dimension of moduli spaces means their virtual dimension.

Consider the moduli space
\[{\mathcal M}_{k+1}\bigl(B;L_{i_0}, \dots, L_{i_k}; R_{\alpha_0}, \dots R_{\alpha_k}\bigr)\]
of bordered stable maps of genus 0, with connected boundary and $(k+1)$ boundary marked points, representing the class $B$.

Set ${\mathcal L}=\bigl(L_{i_0}, \dots, L_{i_k}\bigr)$ and ${\mathcal R}=\bigl(R_{\alpha_0}, \dots R_{\alpha_k}\bigr)$ and write
\[
{\mathcal M}_{k+1}(B;{\mathcal L};{\mathcal R})={\mathcal M}_{k+1}\bigl(B;L_{i_0}, \dots, L_{i_k}; R_{\alpha_0}, \dots R_{\alpha_k}\bigr).
\]
Denote by ${\rm ev}_j^B\colon {\mathcal M}_{k+1}(B;{\mathcal L}; {\mathcal R}) \to R_{\alpha_j}$ the evaluation map at $z_j$.

For a pair of Lagrangian submanifolds $L$, $L'$ which intersect cleanly, we constructed
the $O(1)$-local system $\Theta^-_{R_\alpha}$ on $R_{\alpha}$ in~\cite[Proposition 8.1.1]{FOOO09II}.
Here $R_{\alpha}$ is a connected component of $L \cap L'$.
In this note, we simply write it as $\Theta_{R_{\alpha}}$.

We recall the construction of $\Theta_{R_{\alpha}}$ briefly. We assume that $L$, $L'$ are equipped with spin structures. In the case of a relative spin pair, we take
$TX \oplus (V \otimes {\mathbb C})$ (on the 3-skeleton of~$X$) instead of~$TX$ and $TL \oplus V$ (resp.\ $TL' \oplus V$) (on the 2-skeleton of $L$ (resp.\ $L'$)
instead of~$TL$, (resp.~$TL'$).
Here $V$ is an oriented real vector bundle on the 3-skeleton of $X$ such that the restriction of~$w_2(V)$ to the 2-skeleton of $L$ (resp.\ $L'$) coincides
$w_2(TL)$ (resp.\ $w_2(TL')$).
The~relative spin structure with the background $V$ is a choice of spin structure of $V \oplus TL$, (resp.\ $V \oplus TL'$).
Then the argument goes in the same way.
See the proof of~\cite[Theorem 8.1.1]{FOOO09II}.
For a point $p$ in the self clean intersection of a Lagrangian immersion $i\colon \widetilde{L} \to X$, there are two local branches of the Lagrangian immersion,
i.e., $i_*\bigl(T_{p'}\widetilde{L}\bigr)$ and $i_*\bigl(T_{p''}\widetilde{L}\bigr)$ where $p', p'' \in \widetilde{L}$ with $p=i(p')=i(p'')$.
Then we run the argument below by replacing $T_p L$ and $T_p L'$ by \smash{$i_*\bigl(T_{p'}\widetilde{L}\bigr)$} and \smash{$i_*\bigl(T_{p''}\widetilde{L}\bigr)$}, respectively.

As written in~\cite[Section 8.8]{FOOO09II}, we consider the space ${\mathcal P}_{R_\alpha} (TL, TL')$ of paths of oriented Lagrangian subspaces in $T_pX$, $p \in R_{\alpha}$,
of the form $R_{\alpha} \oplus \lambda (t)$ such that $R_{\alpha} \oplus \lambda (0) = T_p L$ and $R_{\alpha} \oplus \lambda (1) = T_p L'$.
Here $\lambda$ is regarded as a path of Lagrangian subspaces in
\[V_{R_{\alpha}}=\bigl(T_p L + T_p L'\bigr)/ \bigl(T_p L + T_p L'\bigr)^{\perp_{\omega}} = \bigl(T_p L + T_p L'\bigr)/\bigl(T_p L \cap T_p L'\bigr),
\]
which is
a symplectic vector space. Pick a compatible complex structure on it and consider the Dolbeault operator $\overline{\partial}_{\lambda}$ on $Z_- = \bigl(D^2 \cap \{\operatorname{Re} z \leq 0 \}\bigr) \cup ([0, \infty) \times [0,1])$.

We set $\mu(R_{\alpha};\lambda) = {\rm Index~} \overline{\partial}_{\lambda}$.
The parity of $\mu(R_{\alpha}; \lambda)$ is independent of the choice of $\lambda$ above, since $\lambda \oplus T_p R_{\alpha}$ is a path
of oriented totally real subspaces of $T_pX$ with fixed end points,
$T_p L$, $T_p L'$, $p \in R_{\alpha}$ which are oriented.
Denote by $\mu(R_{\alpha}) = \mu(R_{\alpha};\lambda) \mod 2$.
Then we have
\begin{equation}\label{dim moduli}
\dim  {\mathcal M}_{k+1}(B;{\mathcal L}, {\mathcal R}) \equiv \dim  R_{\alpha_0} + \mu(R_{\alpha_0}) - \sum_{i=1}^k \mu\bigl(R_{\alpha_i}\bigr) +k -2 \mod 2.
\end{equation}
We have the determinant line bundle of \smash{$\big\{{\rm Index~} \overline{\partial}_{\lambda} \big\}_{\lambda \in {\mathcal P}_{R_\alpha} (TL, TL')}$}.
Pick a hermitian metric~on~$X$. Denote by $P_{\rm SO}(T_p R_{\alpha} \oplus \lambda) $ the associated oriented orthogonal frame bundle of
$T_p R_{\alpha} \oplus \lambda$.
Note that $P_{\rm SO}(T_p R_{\alpha} \oplus \lambda)\vert_{t=0}$ and $P_{\rm SO}(T_p R_{\alpha} \oplus \lambda)\vert_{t=1}$ are canonically identified with $P_{\rm SO}(L)\vert_p$ and
$P_{\rm SO}(L')\vert_p$, respectively.
We glue the principal spin bundle $P_{\rm Spin}(T_p R_{\alpha} \oplus \lambda)$ at $t=0, 1$ with $P_{\rm Spin}(L)\vert_p$ and $P_{\rm Spin}(L')\vert_p$.
There are two isomorphic classes of resulting spin structure on the bundle $TL \cup (\lambda \oplus T_p R_{\alpha}) \cup TL'$ on $L \cup [0,1] \cup L'$, where
$p \in L$ and $p \in L'$ are identified with ${0, 1 \in [0,1]}$, respectively.
This gives an $O(1)$-local system $O_{\rm Spin}$ on ${\mathcal P}_{R_{\alpha}}(TL, TL')$.
Proposition~8.1.1 in~\cite{FOOO09II} states that the tensor product $\det \overline{\partial}_{\lambda} \otimes O_{\rm Spin}$ descends to an $O(1)$-local system~$\Theta_{R_{\alpha}}$
on~$R_{\alpha}$.

We denote by
$\overline{\partial}_{R_{\alpha}}$ is the Dolbeault operator acting on sections of the trivial bundle
$Z_- \times (T_pR_{\alpha} \otimes {\mathbb C})$ on $Z_-$ with totally real boundary
condition $T_pR_{\alpha}$. Then the operator
$\overline{\partial}_{R_{\alpha_i} \oplus \lambda_i} = \overline{\partial}_{R_{\alpha_i}} \oplus \overline{\partial}_{\lambda_i}$
is the Dolbeault operator acting on the trivial bundle $Z_- \times T_p X$ on $Z_-$ with the totally real boundary condition $T_p R_{\alpha} \oplus \lambda$.
After gluing the linearization operator \smash{$D \overline{\partial}$} for a~holomorphic polygon with
\smash{$\overline{\partial}_{R_{\alpha_i} \oplus \lambda_i}$}, where
$R_{\alpha_i} \oplus \lambda_i \in {\mathcal P}_{R_{\alpha_i}} (TL_{i-1}, TL_i)$, $i=0, \dots, k$, we obtain a Cauchy--Riemann type operator
on the unit disk.
By~\cite[Theorem 8.1.1]{FOOO09II}, the relative spin structure for $\{L_i\}$, namely relative spin structures for each $L_i$ with a common oriented vector bundle~${V \to X^{[3]}}$, determines an isomorphism $\Phi^B$ below.
For the definition and properties of relative spin structure, see~\cite[Section 8.1.1]{FOOO09II}.

\begin{prop}[cf.\ {\cite[Theorem 8.1.1]{FOOO09II}}]
A choice of relative spin structure determines the following isomorphisms.
\begin{enumerate}\itemsep=0pt
\item[$(1)$] Case that $k=0$ $(L$ is an immersed Lagrangian submanifold with clean self intersection or $R_{\alpha_0}=L)$:
\begin{gather*}
\Phi^B\colon \ {\rm ev}_0^{B*} \Theta_{R_{\alpha_0}} \to
{\rm ev}_0^{B*} O_{R_{\alpha_0}}^* \otimes O_{{\mathcal M}_{1}(B;L)}.
\end{gather*}
\item[$(2)$]
Case that $k = 1$:
\begin{gather*}
\Phi^B\colon \ {\rm ev}_0^{B*} \Theta_{R_{\alpha_0}}  \to
{\rm ev}_0^{B*} O_{R_{\alpha_0}}^* \otimes O_{{\mathcal M}_{2}(B;{\mathcal L};{\mathcal R})} \otimes {\mathbb R}_B \otimes
{\rm ev}_1^{B*} \Theta_{R_{\alpha_1}}  \\
\hphantom{\Phi^B\colon \ {\rm ev}_0^{B*} \Theta_{R_{\alpha_0}}\to}{}
 \cong (-1)^{\mu_{\alpha_1}}
 {\rm ev}_0^{B*} O_{R_{\alpha_0}}^* \otimes O_{{\mathcal M}_{2}(B;{\mathcal L};{\mathcal R})} \otimes
{\rm ev}_1^{B*} \Theta_{R_{\alpha_1}} \otimes {\mathbb R}_B .
\end{gather*}
\item[$(3)$] Case that $k \geq 2$:
\begin{gather*}
\Phi^B\colon \ {\rm ev}_0^{B*} \Theta_{R_{\alpha_0}} \to
{\rm ev}_0^{B*} O_{R_{\alpha_0}}^* \otimes O_{{\mathcal M}_{k+1}(B;{\mathcal L};{\mathcal R})} \otimes {\mathfrak{forget}}^* O_{{\mathcal M}_{k+1}}^*\\
\hphantom{\Phi^B\colon \ {\rm ev}_0^{B*} \Theta_{R_{\alpha_0}} \to}{}
\otimes
{\rm ev}_1^{B*} \Theta_{R_{\alpha_1}} \otimes \dots \otimes {\rm ev}_k^{B*} \Theta_{R_{\alpha_k}}.
\end{gather*}
\end{enumerate}
\end{prop}

In item (1), we suppress the orientation bundle of the biholomorphic automorphism group ${\rm Aut}\bigl(D^2, 1\bigr)$, since ${\rm Aut}\bigl(D^2, 1\bigr)$ is
two-dimensional and does not affect the sign when we exchange ${\rm Aut}\bigl(D^2, 1\bigr)$ with other factors.
In item~(2), ${\mathbb R}_B$ is the group of translations in the domain $D^2 \setminus \{\pm 1\} \cong {\mathbb R} \times [0,1]$ and
${\mathcal M}_{2}(B;{\mathcal L};{\mathcal R})$ is the quotient of
\smash{$\widetilde{\mathcal M}_{2}(B;{\mathcal L};{\mathcal R})$} by the translation action of ${\mathbb R}_B$ on the domain,
\[
\widetilde{\mathcal M}_{2}(B;{\mathcal L};{\mathcal R}) = {\mathcal M}_{2}(B;{\mathcal L};{\mathcal R}) \times {\mathbb R}_B.
\]
The sign of the exchange of ${\mathbb R}_B$ and the index of \smash{$\overline{\partial}_{\lambda_{R_{\alpha_1}}}$} is $(-1)^{\mu_{\alpha_1}}$.
In item~(3), ${\mathcal M}_{k+1}$ is the moduli space of bordered Riemann surfaces of genus 0, connected boundary and
$(k+1)$ marked points on the boundary
and $\mathfrak{forget}\colon {\mathcal M}_{k+1}(B;{\mathcal L};{\mathcal R}) \to {\mathcal M}_{k+1}$ sends $[(\Sigma, \partial \Sigma, \vec{z}),u]$
to $[(\Sigma, \partial \Sigma, \vec{z})]$.
Here~$O_{R_{\alpha_0}}$, $O_{{\mathcal M}_{k+1}(B;{\mathcal L};{\mathcal R})}$ and~$O_{{\mathcal M}_{k+1}}$ are orientation bundles of $R_{\alpha_0}$,
${\mathcal M}_{k+1}(B;{\mathcal L};{\mathcal R})$ and ${\mathcal M}_{k+1}$, respectively.
We consider \smash{${\rm ev}_0^* O_{R_{\alpha_0}}^* \otimes O_{{\mathcal M}_{k+1}(B;{\mathcal L};{\mathcal R})}$} the orientation bundle of the relative
tangent bundle of ${\rm ev}_0\colon {\mathcal M}_{k+1}(B;{\mathcal L};{\mathcal R}) \to R_{\alpha_0}$.
In the notation in~\cite{FOOO09II}, we write
\[ {\mathcal M}_{k+1}(B;{\mathcal L};{\mathcal R}) = R_{\alpha_0} \times {}^{\circ} {\mathcal M}_{k+1}(B;{\mathcal L};{\mathcal R})\]
and
\[{\mathcal M}_{k+1}(B;{\mathcal L};{\mathcal R}) = {\mathcal M}_{k+1}(B;{\mathcal L};{\mathcal R})^{\circ} \times {\mathcal M}_{k+1}.\]
These descriptions are considered as the splitting of tangent spaces in the sense of Kuranishi structures.
One may consider $ {}^{\circ} {\mathcal M}_{k+1}(B;{\mathcal L};{\mathcal R})$ and ${\mathcal M}_{k+1}(B;{\mathcal L};{\mathcal R})^{\circ}$
as a fiber of ${\rm ev}_0$ and a fiber of $\mathfrak{forget}$, respectively.
Using these notations, we have
\begin{gather*}
 {\rm ev}_0^* O_{R_{\alpha_0}}^* \otimes O_{{\mathcal M}_{k+1}(B;{\mathcal L};{\mathcal R})}
= O_{ {}^{\circ} {\mathcal M}_{k+1}(B;{\mathcal L};{\mathcal R})},\\
O_{{\mathcal M}_{k+1}(B;{\mathcal L};{\mathcal R})} \otimes {\mathfrak{forget}}^* O_{{\mathcal M}_{k+1}}^* = O_{{\mathcal M}_{k+1}(B;{\mathcal L};{\mathcal R})^{\circ}}.
\end{gather*}
If we denote by ${\mathcal M}_2$ the quotient stack of a point by ${\mathbb R}_B$, (2) is written in (3) with $k=2$.
We give an orientation of \smash{${\mathcal M}_{k+1} = \bigl(\partial D^2\bigr)^{k+1}/{\rm Aut}\bigl(D^2, \partial D^2\bigr)$} as the orientation of the quotient space following~\cite[convention (8.2.1.2)]{FOOO09II}.
Then the orientation bundle of $ {\mathcal M}_{k+1}(B;{\mathcal L};{\mathcal R})$ is canonically isomorphic to the one of $ {\mathcal M}_{k+1}(B;{\mathcal L};{\mathcal R})^{\circ}$.
Hence, for
$
{\mathbf u}=\bigl[u\colon (\Sigma, \partial \Sigma, \vec{z}) \to \smash{\bigl(X, \bigcup_{L \in \mathcal L} L, \bigcup_{R_{\alpha} \in \mathcal R} R_{\alpha}\bigr)}\bigr]$,
the relative spin structure of $\mathcal L$, local sections $\sigma_{\alpha_i}$ of $O(1)$-local systems $\Theta_{\alpha_i}$ around $u(z_i)$, ${i = 0, 1, \dots, k}$,
determines a local orientation of the relative tangent bundle of
\smash{${\rm ev}_0^B\colon {\mathcal M}_{k+1}(B;{\mathcal L};{\mathcal R}) \to R_{\alpha}$},
 at $\mathbf u$, i.e.,
the kernel of \smash{$T_{\mathbf u} {\mathcal M}_{k+1}(B;{\mathcal L};{\mathcal R}) \to T_{u(z_0)}R_{\alpha_0}$}, which is denoted by ${\mathit o}(\sigma_{\alpha_0}; \sigma_{\alpha_1}, \dots, \sigma_{\alpha_k})$.

\begin{rem}
When $k=0$ and $R_{\alpha_0}=L$, the orientation on ${\mathcal M}_{1}(B;L)$ is given in~\cite[Section~8.4.1]{FOOO09II}
When $k=1$, the orientation bundle of ${\mathcal M}_{2}(B;{\mathcal L};{\mathcal R})$ is given in~\cite[Proposition~8.8.6]{FOOO09II}.
Note that $\Theta^+_{R_{\alpha}} \otimes O_{R_{\alpha}} \otimes \Theta^-_{R_{\alpha}}$ is canonically trivialized.
We write $\Theta_{R_{\alpha}}=\Theta_{R_{\alpha}}^-$ in this note.
\end{rem}

Since the evaluation maps are weakly submersive in the sense of Kuranishi structure, see \mbox{\cite[Definition~3.40\,(5)]{FOOO20}},
i.e., after taking sufficiently large obstruction bundles,
the evaluation maps on Kuranishi charts are submersive, the push-forward $({\rm ev}_0)_!$ is defined by taking CF-perturbations.
Hence, for a smooth correspondence $({\mathcal M}_{k+1}(B;{\mathcal L};{\mathcal R}), {\rm ev}_0, {\rm ev}_1 \times \dots \times {\rm ev}_k)$, Theorem 27.1 in~\cite{FOOO20} gives
\begin{equation*}
\bigl({\rm ev}_0^B\bigr)_! \circ \bigl({\rm ev}_1^{B*} \times \dots \times {\rm ev}_k^{B*}\bigr) \colon \ \Omega^*\bigl(R_{\alpha_1};\Theta_{R_{\alpha_1}}\bigr) \otimes \dots \otimes \Omega^*\bigl(R_{\alpha_k};\Theta_{R_{\alpha_k}}\bigr) \to
\Omega^*\bigl(R_{\alpha_0};\Theta_{R_{\alpha_0}}\bigr).
\end{equation*}
Namely, for $\xi_i=\zeta_i \otimes \sigma_{\alpha_i} \in \Omega^*\big(R_{\alpha_i}; \Theta_{\alpha_i}\big)$, $i = 1, \dots, k$, we define
\begin{gather}
\bigl({\rm ev}_0^B\bigr)_! \circ \bigl({\rm ev}_1^{B*} \times \dots \times {\rm ev}_k^{B*}\bigr)\bigl(\zeta_1 \otimes \sigma_{\alpha_1}, \dots, \zeta_k \otimes \sigma_{\alpha_k}\bigr) \nonumber \\
\qquad{} = \bigl({\rm ev}^B_0; {\mathit o}\bigl(\sigma_{\alpha_0}; \sigma_{\alpha_1}, \dots, \sigma_{\alpha_k}\bigr)\bigr)_! \bigl({\rm ev}_1^{B*} \zeta_1 \wedge \dots \wedge {\rm ev}_k^{B*}\zeta_k\bigr) \otimes \sigma_{\sigma_{\alpha_0}}.\label{pre m_k}
\end{gather}
Here $\bigl({\rm ev}^B_0;{\mathit o}\bigl(\sigma_{\alpha_0}; \sigma_{\alpha_1}, \dots, \sigma_{\alpha_k}\bigr)\bigr)_! $ is the integration along fibers with respect to the relative orientation
${\mathit o}\bigl(\sigma_{\alpha_0}; \sigma_{\alpha_1}, \dots, \sigma_{\alpha_k}\bigr)$ of ${\mathcal M}_{k+1}(B; {\mathcal L}; {\mathcal R}) \to R_{\alpha_0}$.
Note that the right hand side of~\eqref{pre m_k} does not depends on $\sigma_{\alpha_0}$, since $\sigma_{\alpha_0}$ appears twice in the right hand side of~\eqref{pre m_k}, and
gives a~differential form on $R_{\alpha_0}$ with coefficients in $\Theta_{\alpha_0}$.
For general $\xi_i \in \Omega^*\bigl(R_{\alpha_i}; \Theta_{\alpha_i}\bigr)$, we use partitions of unity on $R_{\alpha_i}$ and extend the operation
$\bigl({\rm ev}_0^B\bigr)_! \circ \bigl({\rm ev}_1^{B*} \times \dots \times {\rm ev}_k^{B*}\bigr)$ multi-linearly.

For $\xi \in \Omega^*(R_{\alpha};\Theta_{\alpha})$, we define the shifted degree
\begin{equation}\label{degree}
\vert \xi \vert' = {\rm deg~ }\xi + \mu(R_{\alpha}) -1.
\end{equation}

\begin{defn}\label{def} We set
${\mathfrak m}_{0,0}=0$,
${\mathfrak m}_{(1,0)} \xi = {\rm d} \xi$ on $\bigoplus \Omega^*(R_{\alpha}; \Theta_{R_\alpha})$, i.e.,
the de Rham differential on differential forms with coefficients in the local system $\Theta_{R_\alpha}$.
For $(k,B) \neq (1,0)$,
\begin{gather*}
{\mathfrak m}_{k,B}(\xi_1, \dots, \xi_k) \\
\qquad{}= (-1)^{\epsilon(\xi_1, \dots, \xi_k)} \bigl({\rm ev}^B_0\bigr)_! \circ \bigl({\rm ev}_1^{B*} \times \dots \times {\rm ev}_k^{B*}\bigr) (\xi_1 \otimes \dots,\otimes \xi_k) \in \Omega^*(R_{\alpha_0};\Theta_{\alpha_0}),
\end{gather*}
where $\xi_i \in \Omega^*\bigl(R_{\alpha_i};\Theta_{\alpha_i}\bigr)$ and
\begin{equation}\label{epsilon}
\epsilon(\xi_1, \dots, \xi_k) = \Biggl\{ \sum_{i=1}^k \Biggl( i+\sum_{p=1}^{i-1}\mu \bigl( R_{\alpha_p} \bigr) \Biggr) (\deg \xi_i -1) \Biggr\} + 1.
\end{equation}
Then we define
\begin{gather*}
{\mathfrak m}_k=\sum_B {\mathfrak m}_{k,B} T^{\langle \omega, B \rangle} \colon \
\bigotimes_{i=1}^k \Omega^*\bigl(R_{\alpha_i};\Theta_{\alpha_i} \otimes \Lambda_0\bigr)\bigl[1 - \mu\bigl(R_{\alpha_i}\bigr)\bigr]\\
\hphantom{{\mathfrak m}_k=\sum_B {\mathfrak m}_{k,B} T^{\langle \omega, B \rangle} \colon} \
\to \Omega^*(R_{\alpha_0};\Theta_{\alpha_0} \otimes \Lambda_0)[1 - \mu(R_{\alpha_0})].
\end{gather*}
Here
\[\Lambda_0 = \biggl\{ \sum_i a_i T^{\lambda_i} \mid a_i \in {\mathbb R},\, \lambda_i \to \infty \ {\rm as} \ i \to \infty\biggr\}\]
and the symbol $[1 - \mu(R_{\alpha})]$ is the degree shift by $1 - \mu(R_{\alpha})$, i.e., the grading of a differential form is given by $\vert \xi \vert' $.
By~\eqref{deg pushout} and~\eqref{dim moduli}, we find that
\begin{equation}\label{deg m_k}
\vert {\mathfrak m}_k(\xi_1, \dots, \xi_k) \vert' \equiv \sum_{i=1}^k \vert \xi_i \vert' + 1 \mod 2.
\end{equation}
\end{defn}
\begin{rem}
Since the aim of this note is describe the sign and orientation for the filtered $A_{\infty}$-operations, we use $\Lambda_0$ as the coefficient ring.
To make ${\mathfrak m}_k$ operations of degree 1, we need to use the universal Novikov ring $\Lambda_{0,{\rm nov}}$ introduced in~\cite{FOOO09I}.
\end{rem}

\section[Filtered A\_{infty}-relations]{Filtered $\boldsymbol{A_{\infty}}$-relations}
 In the rest of this note, we verify the sign convention in the filtered $A_{\infty}$-relations
 \begin{equation*}
\sum_{k' + k'' = k+1} {\mathfrak m}_{k'} \circ \widehat{\mathfrak m}_{k''} (\xi_1, \dots, \xi_k) = 0 \qquad {\rm for} \ k=1,2, \dots
\end{equation*}
under the tree-like K-system and CF-perturbation described in~\cite{FOOO20}.
Here $\hat{\mathfrak m}_{k,B}$ is the extension of~${\mathfrak m}_{k,B}$ as a graded coderivation with respect to the shifted degree $\vert \bullet \vert'$.
This relation is equivalent to the following relations for decompositions of $B$ into $B'$ and $ B''$, $k'+k''=k+1$,
\begin{gather*}
 {\mathfrak m}_{1,0} \circ {\mathfrak m}_{k, B} (\xi_1, \dots, \xi_k) + {\mathfrak m}_{k,B} \circ \hat{\mathfrak m}_{1,0} (\xi_1, \dots, \xi_k)  \\
\qquad{} + \sum_{(k',B'), (k'', B'') \neq (1,0)} {\mathfrak m}_{k',B'} \circ \hat{\mathfrak m}_{k'',B''} (\xi_1, \dots, \xi_k) = 0.
\end{gather*}
We compute ${\mathfrak m}_{k',B'} \circ \hat{\mathfrak m}_{k'', B''}$.
For $(k,B)=(1,0)$, ${\mathfrak m}_{1,0} \circ {\mathfrak m}_{1,0} = 0$ clearly holds.

From now on, we investigate the case that $(k,B) \neq (1,0)$.
Firstly we consider the case that $(k',B')=(1,0)$ or $(k'',B'') = (1,0)$.
We find that
\begin{gather}
{\mathfrak m}_{1,0} \circ {\mathfrak m}_{k,B} (\xi_1, \dots, \xi_k) = (-1)^{\epsilon(\xi_1, \dots, \xi_k)} {\rm d} \bigl({\rm ev}_0^B\bigr)_! \bigl({\rm ev}_1^{B*} \xi_1 \wedge \dots \wedge {\rm ev}_k^{B*}\xi_k\bigr), \label{rel1}\\
{\mathfrak m}_{k,B} \circ \hat{\mathfrak m}_{1,0} (\xi_1, \dots, \xi_k) = \sum_{j=1}^k (-1)^{\sum_{p=1}^{j-1} \vert \xi_p \vert'} {\mathfrak m}_{k,B} (\xi_1, \dots, d\xi_j, \dots, \xi_k)
\nonumber \\ \hphantom{{\mathfrak m}_{k,B} \circ \hat{\mathfrak m}_{1,0} (\xi_1, \dots, \xi_k)}{}
= \sum_{j=1}^k (-1)^{\sum_{p=1}^{j-1} \vert \xi_p \vert' + \epsilon(\xi_1, \dots, d\xi_j, \dots, \xi_k)} \nonumber \\ \hphantom{{\mathfrak m}_{k,B} \circ \hat{\mathfrak m}_{1,0} (\xi_1, \dots, \xi_k)=}{}
 \times \bigl({\rm ev}_0^B\bigr)_!\bigl({\rm ev}_1^{B*}\xi_1 \wedge \dots \wedge {\rm ev}_j^{B*} d\xi_j \wedge \dots
\wedge {\rm ev}_k^{B*}\xi_k\bigr) \nonumber \\ \hphantom{{\mathfrak m}_{k,B} \circ \hat{\mathfrak m}_{1,0} (\xi_1, \dots, \xi_k)}{}
=  (-1)^{\epsilon(\xi_1, \dots, \xi_k) +1} \bigl({\rm ev}_0^B\bigr)_! {\rm d}\bigl({\rm ev}_1^{B*}\xi_1 \wedge \dots \wedge {\rm ev}_k^{B*} \xi_k\bigr). \label{rel2}
\end{gather}
Here we note that
\begin{gather*}
\sum_{p=1}^{j-1} \vert \xi_p \vert' + \epsilon(\xi_1, \dots, {\rm d}\xi_j, \dots, \xi_k)  =  \sum_{p=1}^{j-1} {\rm deg~ } \xi_p + \sum_{p=1}^{j-1} \bigl( \mu\bigl(R_{\alpha_p}\bigr)-1 \bigr) +
\epsilon(\xi_1, \dots, \xi_k)  \\ \hphantom{\sum_{p=1}^{j-1} \vert \xi_p \vert' + \epsilon(\xi_1, \dots, d\xi_j, \dots, \xi_k)  =}{}
  + \Biggl( j + \sum_{p=1}^{j-1} \mu\bigl(R_{\alpha_p}\bigr) \Biggr)  \\ \hphantom{\sum_{p=1}^{j-1} \vert \xi_p \vert' + \epsilon(\xi_1, \dots, d\xi_j, \dots, \xi_k) }{}
 \equiv  \sum_{p=1}^{j-1} {\rm deg~ }\xi_p + \epsilon(\xi_1, \dots, \xi_k) +1 \mod 2.
\end{gather*}

In order to compute ${\mathfrak m}_{k',B'} \circ \hat{\mathfrak m}_{k'', B''}$ for $(k',B'), (k'',B'') \neq (1,0)$,
we discuss the relation between the orientation bundle of
\[
{\mathcal M}_{k'+1}(B';{\mathcal L}';{\mathcal R}')_{{\rm ev}_j^{B'}} \times_{{\rm ev}_0^{B''}} {\mathcal M}_{k''+1}(B'';{\mathcal L}'';{\mathcal R}'')
 \]
 and~the orientation bundle of the boundary of $\partial {\mathcal M}_{k+1}(B;{\mathcal L}; {\mathcal R})$.
 The codimension $1$ boundary of the moduli space ${\mathcal M}_{k+1}(B;{\mathcal L}; {\mathcal R})$ is
the union of the fiber products of
${\mathcal M}_{k'+1}(B';{\mathcal L}';{\mathcal R}')$ and ${\mathcal M}_{k''+1}(B'';{\mathcal L}'';{\mathcal R}'')$ with respect to
the evaluation maps
\smash{${\rm ev}_j^{B'}\colon {\mathcal M}_{k'+1}(B';{\mathcal L}';{\mathcal R}') \to R_{\alpha}$} and
\smash{${\rm ev}_0^{B''}\colon {\mathcal M}_{k''+1}(B'';{\mathcal L}'';{\mathcal R}'') \to R_{\alpha}$},
where
\begin{alignat*}{3}
&{\mathcal L}'=\bigl(L_{i_0}, \dots, L_{i_{j-1}}, L_{i_{j+k''-1}}, \dots, L_{i_k}\bigr),\qquad&& {\mathcal L}''=\bigl(L_{i_{j-1}}, \dots, L_{i_{j+k''-1}}\bigr),& \\
&{\mathcal R}'=\bigl(R_{\alpha_0}, \dots , R_{\alpha_{j-1}}, R_{\alpha}, R_{\alpha_{j+k''}}, \dots, R_{\alpha_k}\bigr),\qquad&& {\mathcal R}''=\bigl(R_{\alpha}, R_{\alpha_{j}}, \dots R_{\alpha_{i_{j+k''-1}}}\bigr).&
\end{alignat*}
Here the union is taken over $k'$, $k''$ such that $k' + k'' = k+1$, all possible decomposition of $B$ into $B'$ and $B''$, $j=1, \dots, k'$, and
$R_{\alpha}$ a connected component of $L_{i_{j-1}} \cap L_{j+k''-1}$,.

 \begin{prop} \label{bdryori}
\begin{equation*}
(-1)^{\kappa} {\mathcal M}_{k'+1}\bigl(B';{\mathcal L}';{\mathcal R}'\bigr)_{{\rm ev}_j^{B'}} \times_{{\rm ev}_0^{B''}} {\mathcal M}_{k''+1}\bigl(B'';{\mathcal L}'';{\mathcal R}''\bigr)
\subset \partial {\mathcal M}_{k+1}(B;{\mathcal L}; {\mathcal R}),
\end{equation*}
where
\begin{gather*}
\kappa \equiv  (k''-1) (k'-j) +(k'-1) \Biggl( \mu(R_{\alpha}) - \sum_{p=j}^{j+k''-1} \mu\bigl(R_{\alpha_p}\bigr) \Biggr)  \\ \hphantom{\kappa\equiv}{}
  + \Biggl( \sum_{p=1}^{j-1}\mu\bigl(R_{\alpha_p}\bigr) \Biggr)
\Biggl(\mu(R_{\alpha}) - \sum_{p=j}^{j+k''-1} \mu\bigl(R_{\alpha_p}\bigr) \Biggr)
 \\ \hphantom{\kappa\equiv}{}
  + \dim R_{\alpha_0} + \mu(R_{\alpha_0}) - \Biggl( \sum_{p=1}^{j-1} \mu\bigl(R_{\alpha_p}\bigr) + \mu(R_{\alpha}) + \sum_{p=j+k''}^k \mu \bigl(R_{\alpha_p}\bigr) \Biggr) +k'.
\end{gather*}
\end{prop}

\begin{proof}
 Denote by $Sw$ the operation, which exchanges
 \[
 \Theta_{R_{\alpha_1}} \otimes \dots \otimes \Theta_{R_{\alpha_{j-1}}} \qquad \text{and} \qquad
 O_{R_{\alpha}}^* \otimes O_{{\mathcal M}_{k''+1}(B'';{\mathcal L}'';{\mathcal R}'')^{\circ}}.
 \]
Set the weight of $\Theta_{R_{\alpha_i}}$, $O_{R_{\alpha}}$ and $O_{{\mathcal M}_{k+1}(B;{\mathcal L};{\mathcal R})^{\circ}}$
as $\mu(R_{\alpha_i})$, $\dim R_{\alpha}$ and
\[
\dim {\mathcal M}_{k+1}(B;{\mathcal L};{\mathcal R})^{\circ} =
\dim {\mathcal M}_{k+1}(B;{\mathcal L};{\mathcal R}) - \dim {\mathcal M}_{k+1},
\]
 respectively.
Then the weighted sign of $Sw$ is $(-1)^{\delta_1}$, where
\begin{align*}
\delta_1  &{}=  \Biggl(\sum_{p=1}^{j-1}\mu\bigl(R_{\alpha_p}\bigr)\Biggr) \bigl(\dim {\mathcal M}_{k''+1}\bigl(B'';{\mathcal L}'';{\mathcal R}''\bigr) - \dim R_{\alpha} - \dim {\mathcal M}_{k''+1} \bigr) \\
 &{}\equiv  \Biggl( \sum_{p=1}^{j-1}\mu\bigl(R_{\alpha_p}\bigr)\Biggr) \Biggl( \mu(R_{\alpha}) - \sum_{p=j}^{j+k''-1}\mu\bigl(R_{\alpha_p}\bigr) \Biggr) \mod 2.
\end{align*}

Comparing $\Phi^B$ and $Sw \circ \bigl({\rm id} \otimes \dots \otimes {\rm id} \otimes \Phi^{B''} \otimes {\rm id} \otimes \dots \otimes {\rm id}\bigr) \circ \Phi^{B'}$, we find that
\[O_{{\mathcal M}_{k+1}(B;{\mathcal L};{\mathcal R})^{\circ}} \to
O_{{\mathcal M}_{k'+1}(B';{\mathcal L}';{\mathcal R}')^{\circ}} \otimes O_{R_{\alpha}}^* \otimes O_{{\mathcal M}_{k''+1}(B'';{\mathcal L}'';{\mathcal R}'')^{\circ}}\]
is $(-1)^{\delta_1}$-orientation preserving.\footnote{ $(-1)$-orientation preserving means orientation reversing. }
Here ${\mathcal M}_{k+1}(B;{\mathcal L};{\mathcal R})^{\circ}$ is the fiber of
${\mathcal M}_{k+1}(B;{\mathcal L};{\mathcal R}) \to {\mathcal M}_{k+1}$, i.e., the moduli space of bordered stable maps
with a fixed domain bordered Riemann surface equipped with fixed boundary marked points.
The $O(1)$-local system
\[
O_{{\mathcal M}_{k'+1}(B';{\mathcal L}';{\mathcal R}')^{\circ}} \otimes O_{R_{\alpha}}^* \otimes O_{{\mathcal M}_{k''+1}(B'';{\mathcal L}'';{\mathcal R}'')^{\circ}}
\]
is the orientation bundle of the fiber product
\[
{\mathcal M}_{k'+1}\bigl(B';{\mathcal L}';{\mathcal R}'\bigr)^{\circ}_{{\rm ev}_j^{B'}} \times_{{\rm ev}_0^{B''}} {\mathcal M}_{k''+1}\bigl(B'';{\mathcal L}'';{\mathcal R}''\bigr)^{\circ},
\]
which is the moduli space of bordered stable maps with a fixed boundary nodal Riemann surface equipped with fixed boundary marked points.

Now we compare the orientations of
\begin{gather*}
\partial {\mathcal M}_{k+1}(B;{\mathcal L};{\mathcal R})= \partial \bigl( {\mathcal M}_{k+1}(B;{\mathcal L};{\mathcal R})^{\circ} \times {\mathcal M}_{k+1} \bigr)
\end{gather*}
and
\begin{gather*}
{\mathcal M}_{k'+1}\bigl(B';{\mathcal L}';{\mathcal R}'\bigr)_{{\rm ev}_j^{B'}} \times_{{\rm ev}_0^{B''}} {\mathcal M}_{k''+1}\bigl(B'';{\mathcal L}'';{\mathcal R}''\bigr)\\
\qquad{}
=\bigl( {\mathcal M}_{k'+1}\bigl(B';{\mathcal L}';{\mathcal R}'\bigr)^{\circ} \times {\mathcal M}_{k'+1} \bigr)_{{\rm ev}_j^{B'}} \times_{{\rm ev}_0^{B''}}
\bigl( {\mathcal M}_{k''+1}\bigl(B'';{\mathcal L}'';{\mathcal R}''\bigr)^{\circ} \times {\mathcal M}_{k''+1} \bigr).
\end{gather*}
We note that $O_{{\mathcal M}_{k+1}(B;{\mathcal L};{\mathcal R}) } = {\mathbb R}_{{\rm out}} \otimes O_{\partial {\mathcal M}_{k+1}(B;{\mathcal L};{\mathcal R})}$.
Here ${\mathbb R}_{\rm out}$ is the normal bundle of the boundary oriented by the outer normal vector.
We pick local flat sections $\sigma_{\alpha_0}, \dots, \sigma_{\alpha_k}, \sigma_{\alpha}$ of $O(1)$-local systems $\Theta_{R_{\alpha_0}}, \dots, \Theta_{R_{\alpha_k}}, \Theta_{R_{\alpha}}$ and a~local orientation ${\mathit o}_{R_{\alpha_0}}$ of $R_{\alpha_0}$ around $u(z_0)$. Then we can equip
${\mathcal M}_{k+1}(B;{\mathcal L};{\mathcal R})$, ${\mathcal M}_{k'+1}(B';{\mathcal L}';{\mathcal R}')$ and the relative tangent bundle of
\[
{\rm ev}_0^{B''}\colon \ \mathcal M_{k''+1}(B'';{\mathcal L}'';{\mathcal R}'') \to R_{\alpha}
\]
with local orientations induced by them.
Then a local orientation of ${\mathcal M}_{k+1}(B;{\mathcal L};{\mathcal R}) = R_{\alpha_0} \times^{\circ}{\mathcal M}_{k+1}(B;{\mathcal L};{\mathcal R})$ is given by
${\mathit o}_{R_{\alpha_0}} \times o\bigl(\sigma_{\alpha_0};\sigma_{\alpha_1}, \dots, \sigma_{\alpha_k}\bigr)$.
As the fiber product of spaces with Kuranishi structures equipped with local orientations,
\begin{gather*}
\begin{split}
& {\mathcal M}_{k'+1}\bigl(B';{\mathcal L}';{\mathcal R}'\bigr)_{{\rm ev}_j^{B'}} \times_{{\rm ev}_0^{B''}} {\mathcal M}_{k''+1}\bigl(B'';{\mathcal L}'';{\mathcal R}''\bigr)\\
& \qquad{}= {\mathcal M}_{k'+1}\bigl(B';{\mathcal L}';{\mathcal R}'\bigr) \times {}^{\circ}{\mathcal M}_{k''+1}\bigl(B'';{\mathcal L}'';{\mathcal R}''\bigr)
\end{split}
\end{gather*}
is locally oriented by
\[o_{R_{\alpha_0}} \times o\bigl(\sigma_{R_{\alpha_0}};\sigma_{R_{\alpha_1}}, \dots, \sigma_{R_{\alpha_{j-1}}}, \sigma_{R_\alpha}, \sigma_{R_{\alpha_{j+k''}}}, \dots, \sigma_{R_{\alpha_k}} \bigr)
\times o\bigl(\sigma_{R_{\alpha}}; \sigma_{R_{\alpha_j}}, \dots, \sigma_{R_{\alpha_{j+k''-1}}}\bigr).\]
We fix $z_0=+1$, $z_j=-1$ and consider the spaces of $J$-holomorphic maps $\widetilde{\mathcal M}_{k+1}(B;{\mathcal L}, {\mathcal R})$,
\smash{$\widetilde{\mathcal M}_{k'+1}(B';{\mathcal L}', {\mathcal R}')$} and \smash{$\widetilde{\mathcal M}_{k''+1}(B'';{\mathcal L},'' {\mathcal R}'')$} such that
\begin{gather*}
{\mathcal M}_{k+1}(B;{\mathcal L};{\mathcal R})=\widetilde{\mathcal M}_{k+1}(B;{\mathcal L}, {\mathcal R})/{\mathbb R}_B,\\
{\mathcal M}_{k'+1}\bigl(B';{\mathcal L}';{\mathcal R}'\bigr)=\widetilde{\mathcal M}_{k'+1}\bigl(B';{\mathcal L}';{\mathcal R}'\bigr)/{\mathbb R}_{B'},
\end{gather*}
and
\[{\mathcal M}_{k''+1}\bigl(B'';{\mathcal L}'';{\mathcal R}''\bigr)=\widetilde{\mathcal M}_{k''+1}\bigl(B'';{\mathcal L},'' {\mathcal R}''\bigr)/{\mathbb R}_{B''}.\]
We may also write
\[\widetilde{\mathcal M}_{k+1}(B;{\mathcal L};{\mathcal R})={\mathcal M}_{k+1}(B;{\mathcal L}, {\mathcal R}) \times {\mathbb R}_B, \qquad {\rm etc.},\]
as oriented spaces.

The case that $z_0=+1$, $z_1 = -1$ is discussed in~\cite[p.~699]{FOOO09II}. The case that $z_0=+1$, $z_j = -1$ differs from the case that $z_0=+1$, $z_1 = -1$ by
an additional factor $(-1)^{j-1}$ as below.

For orientation issue, we consider the top-dimensional strata of the moduli spaces and regard
\smash{$\widetilde{\mathcal M}_{k+1}(B;{\mathcal L};{\mathcal R})$} as an open subset of
\[
{\mathcal M}_{k+1}(B;{\mathcal L};{\mathcal R})^{\circ} \times \prod_{i=1}^{j-1}(\partial D)_{z_i} \times
\prod_{i=j+1}^{j+k''-1}(\partial D)_{z_i} \times \prod_{i=j+k''}^k (\partial D)_{z_i}.
\]
We simply write
\[\widetilde{\mathcal M}_{k+1}(B;{\mathcal L};{\mathcal R})=(-1)^{j-1} {\mathcal M}_{k+1}(B;{\mathcal L};{\mathcal R})^{\circ} \times \prod_{i=1}^{j-1}(\partial D)_{z_i} \times
\prod_{i=j+1}^{j+k''-1}(\partial D)_{z_i} \times \prod_{i=j+k''}^k (\partial D)_{z_i},\]
where $z_0 = +1$, $z_j=-1$,
\[\widetilde{\mathcal M}_{k'+1}\bigl(B';{\mathcal L}';{\mathcal R}'\bigr)= (-1)^{j-1}{\mathcal M}_{k'+1}\bigl(B';{\mathcal L}';{\mathcal R}'\bigr)^{\circ} \times \prod_{i=1}^{j-1}(\partial D)_{z_i} \times
\prod_{i=j+k''}^k (\partial D)_{z_i},\]
where $z'_0=+1$, $z'_j=-1$,
and
\[\widetilde{\mathcal M}_{k''+1}\bigl(B'';{\mathcal L}'';{\mathcal R}''\bigr)= {\mathcal M}_{k''+1}\bigl(B'';{\mathcal L};''{\mathcal R}''\bigr)^{\circ} \times
\prod_{i=j+1}^{j+k''-1}(\partial D)_{z_i},\]
where $z''_0=+1$, $z''_1=-1$.

Note that
\begin{gather*}
(-1)^{j-1} \prod_{i=1}^{j-1}(\partial D)_{z_i} \times \prod_{i=j+1}^{j+k''-1}(\partial D)_{z_i} \times \prod_{i=j+k''}^k (\partial D)_{z_i} = {\mathcal M}_{k+1} \times {\mathbb R}_B,\\
(-1)^{j-1} \prod_{i=1}^{j-1}(\partial D)_{z_i} \times \prod_{i=j+k''}^k (\partial D)_{z_i} = {\mathcal M}_{k'+1} \times {\mathbb R}_{B'}
\end{gather*}
and
\[\prod_{i=j+1}^{j+k''-1}(\partial D)_{z_i} = {\mathcal M}_{k'' + 1} \times {\mathbb R}_{B''}.\]

Marked points of ${\mathcal M}_{k'+1}(B';{\mathcal L}';{\mathcal R}')$ and ${\mathcal M}_{k''+1}(B'';{\mathcal L};''{\mathcal R}'')$
are related to marked points of ${\mathcal M}_{k+1}(B;{\mathcal L};{\mathcal R})$ in the following way.
\begin{gather*}
\bigl(z'_0, \dots, z'_{k'}\bigr) = \bigl(z_0, \dots, z_{j-1}, z'_j, z_{j+k''}, \dots, z_k\bigr),\\
\bigl(z''_0, z''_1,, \dots, z''_{k''}\bigr) = \bigl(z''_0, z_j, \dots, z_{j+k''-1}\bigr).
\end{gather*}
Here \smash{$z'_j$} and \smash{$z''_0$} are identified, i.e., the boundary node of the domain curve of an element in~${\mathcal M}_{k+1}(B;{\mathcal L};{\mathcal R})$.
Then we find that
\begin{gather}
  \widetilde{\mathcal M}_{k+1}(B;{\mathcal L};{\mathcal R})\nonumber\\
\quad{} =
(-1)^{\delta_1} \bigl( {\mathcal M}_{k'+1}\bigl(B';{\mathcal L}';{\mathcal R}'\bigr)^{\circ} ~_{{\rm ev}_j^{B'}} \times_{{\rm ev}_0^{B''}}
{\mathcal M}_{k''+1}\bigl(B'';{\mathcal L};''{\mathcal R}''\bigr)^{\circ} \bigr) \nonumber \\
\qquad{}  \times (-1)^{j-1}\prod_{i=1}^{j-1}(\partial D)_{z_i} \times
\prod_{i=j+1}^{j+k''-1}(\partial D)_{z_i} \times \prod_{i=j+k''}^k (\partial D)_{z_i} \nonumber \\
\quad{} =  (-1)^{\delta_1 + \delta_2} \Biggl( {\mathcal M}_{k'+1}\bigl(B';{\mathcal L}';{\mathcal R}'\bigr)^{\circ} \times (-1)^{j-1} \prod_{i=1}^{j-1}(\partial D)_{z_i} \times
\prod_{i=j+k''}^k (\partial D)_{z_i} \Biggr)
\nonumber \\
\qquad{}  {}\ {}_{{\rm ev}_j^{B'}}  \times_{{\rm ev}_0^{B''}} \Biggl( {\mathcal M}_{k''+1}\bigl(B'';{\mathcal L}'';{\mathcal R}''\bigr)^{\circ} \times \prod_{i=j+1}^{j+k''-1}(\partial D)_{z_i} \Biggr) \nonumber \\
\quad{} =  (-1)^{\delta_1 + \delta_2} \bigl( {\mathcal M}_{k'+1}\bigl(B';{\mathcal L}';{\mathcal R}'\bigr) \times {\mathbb R}_{B'} \bigr)
~_{{\rm ev}_j^{B'}} \times_{{\rm ev}_0^{B''}}
\bigl( {\mathcal M}_{k''+1}\bigl(B'';{\mathcal L}'';{\mathcal R}''\bigr)\! \times {\mathbb R}_{B''} \bigr) \nonumber \\
\quad{} =  (-1)^{\delta_1 + \delta_2 + \delta_3} {\mathbb R}_{B'-B''} \times \bigl( {\mathcal M}_{k'+1}\bigl(B';{\mathcal L}';{\mathcal R}'\bigr)\! ~_{{\rm ev}_j^{B'}} \times_{{\rm ev}_0^{B''}}
{\mathcal M}_{k''+1}\bigl(B'';{\mathcal L};''{\mathcal R}''\bigr) \bigr)
\! \times {\mathbb R}_{B'+B''} \nonumber \\
\quad{} =  (-1)^{\delta_1 + \delta_2 + \delta_3} {\mathbb R}_{\rm out} \times \bigl( {\mathcal M}_{k'+1}\bigl(B';{\mathcal L}';{\mathcal R}'\bigr)\! ~_{{\rm ev}_j^{B'}} \times_{{\rm ev}_0^{B''}}
{\mathcal M}_{k''+1}\bigl(B'';{\mathcal L};''{\mathcal R}''\bigr) \bigr)
\!  \times {\mathbb R}_{B}, \label{ori1}
\end{gather}
where
\begin{gather*}
\delta_2 =  \bigl(k''-1\bigr) \bigl(k' -j\bigr) + \bigl(k' -1\bigr) \bigl( \dim {\mathcal M}_{k''+1}\bigl(B'';{\mathcal L};''{\mathcal R}''\bigr)^{\circ} - \dim R_{\alpha} \bigr), \\
\delta_3 = \dim {\mathcal M}_{k'+1}\bigl(B';{\mathcal L}';{\mathcal R}'\bigr).
\end{gather*}
${\mathbb R}_{B'-B''}$ and ${\mathbb R}_{B'+B''}$ are the oriented lines spanned by $(1, -1), (1, 1) \in {\mathbb R}_{B'} \oplus {\mathbb R}_{B''}$, respectively.
Note that the ordered bases $(1,0), (0,1)$ and $(1, -1), (1,1)$ give the same orientation of ${\mathbb R}_{B'} \oplus {\mathbb R}_{B''}$,
${\mathbb R}_{B'-B''}$ and ${\mathbb R}_{B'+B''}$ are identified with ${\mathbb R}_{\rm out}$ and ${\mathbb R}_B$, respectively.

Here is an explanation of the second equality, i.e., the appearance of $(-1)^{\delta_2}$.
By the convention in~\cite[Section 8.2]{FOOO09II}, we have
\begin{gather*}
 {\mathcal M}_{k'+1}\bigl(B';{\mathcal L}';{\mathcal R}'\bigr)^{\circ} ~_{{\rm ev}_j^{B'}} \times_{{\rm ev}_0^{B''}}
{\mathcal M}_{k''+1}\bigl(B'';{\mathcal L};''{\mathcal R}''\bigr)^{\circ} \nonumber \\
\qquad{} = {\mathcal M}_{k'+1}\bigl(B';{\mathcal L}';{\mathcal R}'\bigr)^{\circ \circ} \times R_{\alpha} \times
{}^{\circ}{\mathcal M}_{k''+1}\bigl(B'';{\mathcal L};''{\mathcal R}''\bigr)^{\circ} \nonumber \\
\qquad{} = {\mathcal M}_{k'+1}\bigl(B';{\mathcal L}';{\mathcal R}'\bigr)^{\circ} \times {}^{\circ}{\mathcal M}_{k''+1}\bigl(B'';{\mathcal L};''{\mathcal R}''\bigr)^{\circ}, \nonumber
\end{gather*}
where
\[{\mathcal M}_{k'+1}\bigl(B';{\mathcal L}';{\mathcal R}'\bigr)^{\circ} = {\mathcal M}_{k'+1}\bigl(B';{\mathcal L}';{\mathcal R}'\bigr)^{\circ \circ} \times R_{\alpha}, \]
and
\[{\mathcal M}_{k''+1}\bigl(B'';{\mathcal L};''{\mathcal R}''\bigr)^{\circ} = R_{\alpha} \times {}^{\circ}{\mathcal M}_{k''+1}\bigl(B'';{\mathcal L};''{\mathcal R}''\bigr)^{\circ}.\]
Using these notations, we have
\begin{gather*}
\bigl( {\mathcal M}_{k'+1}\bigl(B';{\mathcal L}';{\mathcal R}'\bigr)^{\circ} ~_{{\rm ev}_j^{B'}} \times_{{\rm ev}_0^{B''}} \times{\mathcal M}_{k''+1}\bigl(B'';{\mathcal L};''{\mathcal R}''\bigr)^{\circ} \bigr)\\ \quad\qquad{}
 \times
\prod_{i=1}^{j-1}(\partial D)_{z_i} \times \prod_{i=j+1}^{j+k''-1}(\partial D)_{z_i} \times \prod_{i=j+k''}^k (\partial D)_{z_i}  \\ \qquad{}
= (-1)^{\gamma_1} \bigl( {\mathcal M}_{k'+1}\bigl(B';{\mathcal L}';{\mathcal R}'\bigr)^{\circ} ~_{{\rm ev}_j^{B'}} \times_{{\rm ev}_0^{B''}} \times{\mathcal M}_{k''+1}\bigl(B'';{\mathcal L};''{\mathcal R}''\bigr)^{\circ} \bigr)  \\ \quad\qquad{}
 \times
\prod_{i=1}^{j-1}(\partial D)_{z_i} \times \prod_{i=j+k''}^k (\partial D)_{z_i} \times \prod_{i=j+1}^{j+k''-1}(\partial D)_{z_i}  \\ \qquad{}
= (-1)^{\gamma_1} \bigl( {\mathcal M}_{k'+1}\bigl(B';{\mathcal L}';{\mathcal R}'\bigr)^{\circ} \times {}^{\circ}{\mathcal M}_{k''+1}\bigl(B'';{\mathcal L};''{\mathcal R}''\bigr)^{\circ} \bigr) \\ \quad\qquad{}
  \times
\prod_{i=1}^{j-1}(\partial D)_{z_i} \times \prod_{i=j+k''}^k (\partial D)_{z_i} \times \prod_{i=j+1}^{j+k''-1}(\partial D)_{z_i}  \\ \qquad{}
 =  (-1)^{\gamma_1+ \gamma_2 } {\mathcal M}_{k'+1}\bigl(B';{\mathcal L}';{\mathcal R}'\bigr)^{\circ} \times \prod_{i=1}^{j-1}(\partial D)_{z_i} \times
\prod_{i=j+k''}^k (\partial D)_{z_i}  \\ \quad\qquad{}
  \times {}^{\circ}{\mathcal M}_{k''+1}\bigl(B'';{\mathcal L};''{\mathcal R}''\bigr)^{\circ} \times \prod_{i=j+1}^{j+k''-1}(\partial D)_{z_i}  \\ \qquad{}
 =  (-1)^{\gamma_1 + \gamma_2} \Biggl( {\mathcal M}_{k'+1}\bigl(B';{\mathcal L}';{\mathcal R}'\bigr)^{\circ} \times \prod_{i=1}^{j-1}(\partial D)_{z_i} \times\prod_{i=j+k''}^k (\partial D)_{z_i} \Biggr)  \\ \quad\qquad{}
 {}\times _{{\rm ev}_j^{B'}} \times_{{\rm ev}_0^{B''}} \Biggl( {\mathcal M}_{k''+1}\bigl(B'';{\mathcal L}'';{\mathcal R}''\bigr)^{\circ} \times \prod_{i=j+1}^{j+k''-1}(\partial D)_{z_i} \Biggr),
\end{gather*}
where $\gamma_1 = (k''-1) (k' -j)$, i.e., $(-1)^{\gamma_1}$ is the sign of switching marked points $(z_{j+k''}, \dots, z_k)$ and $(z_{j+1}, \dots, z_{j+k''-1})$, and
$\gamma_2 = \dim ( {}^{\circ}{\mathcal M}_{k''+1}(B'';{\mathcal L};''{\mathcal R}'')^{\circ} ) ( \dim {\mathcal M}_{k'+1} +1 )$.
Then $\delta_2 = \gamma_1 + \gamma_2$.

Now we return to the discussion on local orientations of the orientation bundle of
\[
{\mathcal M}_{k'+1}(B';{\mathcal L}';{\mathcal R}')_{{\rm ev}_j^{B'}} \times_{{\rm ev}_0^{B''}} {\mathcal M}_{k''+1}(B'';{\mathcal L}'';{\mathcal R}'')
\qquad
\text{and} \qquad
\partial {\mathcal M}_{k+1}(B;{\mathcal L}; {\mathcal R}).
\]
Recall that
\begin{equation}\label{ori2}
\widetilde{\mathcal M}_{k+1}(B;{\mathcal L};{\mathcal R})={\mathcal M}_{k+1}(B;{\mathcal L}, {\mathcal R}) \times {\mathbb R}_B.
\end{equation}
Set $\kappa = \delta_1 + \delta_2 + \delta_3$, i.e.,
\begin{gather*}
\kappa \equiv \bigl(k''-1\bigr) \bigl(k'-j\bigr) +\bigl(k'-1\bigr) \Biggl( \mu(R_{\alpha}) - \sum_{p=j}^{j+k''-1} \mu\bigl(R_{\alpha_p}\bigr) \Biggr) \\ \hphantom{\kappa \equiv}{}
 + \Biggl( \sum_{p=1}^{j-1}\mu\bigl(R_{\alpha_p}\bigr) \Biggr)
\Biggl(\mu(R_{\alpha}) - \sum_{p=j}^{j+k''-1} \mu\bigl(R_{\alpha_p}\bigr) \Biggr)
 \\ \hphantom{\kappa \equiv}{}
 + \dim R_{\alpha_0} + \mu(R_{\alpha_0}) - \Biggl( \sum_{p=1}^{j-1} \mu\bigl(R_{\alpha_p}\bigr) + \mu(R_{\alpha}) + \sum_{p=j+k''}^k \mu \bigl(R_{\alpha_p}\bigr) \Biggr) +k'.
\end{gather*}

Comparing~\eqref{ori1} and~\eqref{ori2}, we obtain Proposition \ref{bdryori}.
\end{proof}

From Corollary \ref{iteratefib} in the setting of Kuranishi structures, Propositions~\ref{bdryori} and~\ref{base change}, i.e., the base change formula for
integration along fibers, we find the following.

\begin{lem}\label{compformula}
\begin{gather*}
\bigl( {\rm ev}_0^B \vert_{\partial {\mathcal M}_{k+1}(B;{\mathcal L};{\mathcal R})}; \partial o\bigl(\sigma_{\alpha_0};\sigma_{\alpha_1}, \dots, \sigma_{\alpha_k}\bigr) \bigr)_!
\Biggl(\prod_{i=1}^{j-1} {\rm ev}_i^{B*} \times \prod_{i=j+k''}^k {\rm ev}_i^{B*} \times \prod_{i=j}^{j+k''-1}{\rm ev}_i^{B*} \Biggr) \\
\quad{} = (-1)^{\kappa} \bigl( {\rm ev}_0^{B'};o\bigl(\sigma_{\alpha_0}; \sigma_{\alpha_1}, \dots, \sigma_{\alpha_{j-1}}, \sigma_{\alpha}, \sigma_{\alpha_{j+k''-1}}, \dots, \sigma_{\alpha_k}\bigr) \bigr)_!  \\
\qquad{} \circ \Biggl( \prod_{i=1}^{j-1} {\rm ev}_i^{B'*} \times \prod_{i=j+1}^{k'} {\rm ev}_i^{B'*} \times \Biggl( {\rm ev}_j^{B'*} \circ \bigl({\rm ev}_0^{B''}; o\bigl(\sigma_{\alpha}; \sigma_{\alpha_j}, \dots,
\sigma_{\alpha_{j+k''-1}}\bigr) \bigr)_! \circ \prod_{i=1}^{k''}{\rm ev}_i^{B''*} \Biggl)\! \Biggr)
\end{gather*}
as operations applied to
\[
\Biggl(\bigotimes_{i=1}^{j-1} \zeta_i\Biggr) \otimes \Biggl(\bigotimes_{i=j+k''}^k \zeta_i\Biggr) \otimes \Biggl(\bigotimes_{i=j}^{j+ k''-1} \zeta_i\Biggr),
\]
where
$\xi_i=\zeta_i \otimes \sigma_{\alpha_i}$, $i=1, \dots, k$.
Here $\partial o\bigl(\sigma_{\alpha_0};\sigma_{\alpha_1}, \dots, \sigma_{\alpha_k}\bigr)$ is the local orientation of the relative tangent bundle
$\partial {\mathcal M}_{k+1}(B;{\mathcal L}; {\mathcal R}) \to R_{\alpha_0}$ induced from
$o\bigl(\sigma_{\alpha_0};\sigma_{\alpha_1}, \dots, \sigma_{\alpha_k}\bigr)$.
\end{lem}

Note that $\partial o\bigl(\sigma_{\alpha_0};\sigma_{\alpha_1}, \dots, \sigma_{\alpha_k}\bigr)$
is not the boundary orientation of $\partial {}^{\circ}{\mathcal M}_{k+1}(B; {\mathcal L}; {\mathcal R})$ induced from the orientation
$o\bigl(\sigma_{\alpha_0};\sigma_{\alpha_1}, \dots, \sigma_{\alpha_k}\bigr)$ of ${}^{\circ}{\mathcal M}_{k+1}(B; {\mathcal L}; {\mathcal R})$.
They differ by $(-1)^{\dim  R_{\alpha_0}}$.
Namely, for ${\mathbf u} \in \partial {\mathcal M}_{k+1}(B;{\mathcal L}; {\mathcal R})$, the local orientation
$o\bigl(\sigma_{\alpha_0};\sigma_{\alpha_1}, \dots, \sigma_{\alpha_k}\bigr)$ of
${}^{\circ}{\mathcal M}_{k+1}(B; {\mathcal L}; {\mathcal R})$ and the local orientation
$\partial o\bigl(\sigma_{\alpha_0};\sigma_{\alpha_1}, \dots, \sigma_{\alpha_k}\bigr)$
of the relative tangent bundle of $\partial {\mathcal M}_{k+1}(B;{\mathcal L}; {\mathcal R})\allowbreak \to R_{\alpha_0}$
are related as follows:
\begin{gather*}
T_{\mathbf u} {\mathcal M}_{k+1}(B;{\mathcal L}; {\mathcal R}) = {\mathbb R}_{\rm out} \times T_{\mathbf u} \partial {\mathcal M}_{k+1}(B;{\mathcal L}; {\mathcal R}),\\
T_{\mathbf u} {\mathcal M}_{k+1}(B;{\mathcal L}; {\mathcal R}) = T_{u(z_0)} R_{\alpha_0} \times T_{\mathbf u} {}^{\circ}{\mathcal M}_{k+1}(B; {\mathcal L}; {\mathcal R}).
\end{gather*}
Then, under the following identification
\[{\mathbb R}_{\rm out} \times T_{\mathbf u} \partial {\mathcal M}_{k+1}(B;{\mathcal L}; {\mathcal R}) ={\mathbb R}_{\rm out} \times T_{u(z_0)} R_{\alpha_0} \times
T_{\mathbf u} {}^{\circ} \partial {\mathcal M}_{k+1}(B; {\mathcal L}; {\mathcal R}),\]
we define the local orientation $\partial o(\sigma_{\alpha_0};\sigma_{\alpha_1}, \dots, \sigma_{\alpha_k})$ of the relative tangent bundle of
$\partial {\mathcal M}_{k+1}(B;\allowbreak {\mathcal L}; {\mathcal R}) \to R_{\alpha_0}$
so that
\[
o_{R_{\alpha_0}} \times o(\sigma_{\alpha_0};\sigma_{\alpha_1}, \dots, \sigma_{\alpha_k}) = {\mathbb R}_{\rm out} \times o_{R_{\alpha_0}} \times
\partial o(\sigma_{\alpha_0};\sigma_{\alpha_1}, \dots, \sigma_{\alpha_k}).
\]
Note that
\begin{equation*}
{\rm ev}_i^B\vert_{\partial {\mathcal M}_{k+1}(B; {\mathcal L}; {\mathcal R})}=
\begin{cases}
{\rm ev}_i^{B'} \circ \pi^B_{B'}, & i=1, \dots, j-1, \\
{\rm ev}_{i-j+1}^{B''} \circ \pi^B_{B''}, & i=j, \dots, j+k''-1, \\
{\rm ev}_{i-k''+1}^{B'} \circ \pi^B_{B'},& i=j+k'', \dots, k,
\end{cases}
\end{equation*}
where
$\pi^B_{B'}$ and $\pi^B_{B''}$ are projections from the fiber product
\[{\mathcal M}_{k'+1}\bigl(B';{\mathcal L}';{\mathcal R}'\bigr)_{{\rm ev}_j^{B'}} \times_{{\rm ev}_0^{B''}} {\mathcal M}_{k''+1}\bigl(B'';{\mathcal L}'';{\mathcal R}''\bigr)\]
to ${\mathcal M}_{k'+1}(B';{\mathcal L}';{\mathcal R}')$ and ${\mathcal M}_{k''+1}(B'';{\mathcal L}'';{\mathcal R}'')$, respectively.
Note that $\sigma_{\alpha}$ appears twice in the right hand side of the equality in Lemma \ref{compformula}, hence the right hand side does not depends
on the choice of local section $\sigma_{\alpha}$ of the $O(1)$-local system $\Theta_{\alpha}$.

Next, we compute ${\mathfrak m}_{k', B'} \circ \hat{\mathfrak m}_{k'',B''}$ with $(k',B') \neq (1,0)$, $(k'',B'') \neq (1,0)$.
Armed with Lemma~\ref{compformula}, we regard $\xi_i$, $i = 1, \dots, k$, as differential forms on $R_{\alpha_i}$ in the computation below.

\begin{lem}\label{lemma}
\begin{equation}\label{lem}
{\mathfrak m}_{k', B'} \circ \hat{\mathfrak m}_{k'',B''} (\xi_1, \dots, \xi_k) = (-1)^{\kappa'} \bigl({\rm ev}_0^{(B',B'')}\bigr)_! \bigl({\rm ev}_1^{(B',B'')*} \xi_1 \wedge \dots \wedge {\rm ev}_k^{(B',B'')*} \xi_k\bigr),
\end{equation}
where
\begin{gather*}
\kappa' \equiv \epsilon(\xi_1, \dots, \xi_k) + \sum_{i=1}^{k} \deg \xi_i - k - 1 + j + k' \Biggl( \mu(R_{\alpha}) - \sum_{i=j}^{j+k''-1} \mu\bigl(R_{\alpha_i}\bigr) \Biggr)
 \\ \hphantom{\kappa' \equiv}{}
 +  \Biggl(\sum_{p=1}^{j-1} \mu\bigl(R_{\alpha_p}\bigr)\Biggr)\Biggl(\mu(R_{\alpha}) - \sum_{p=j}^{j+k''-1} \mu\bigl(R_{\alpha_p}\bigr)\Biggr) + \bigl(k'-j\bigr)k'' \mod 2.
\end{gather*}
\end{lem}

\begin{proof}
By the definition of ${\mathfrak m}_{k,B}$ and its extension $\widehat{\mathfrak m}_{k,B}$ as a graded coderivation, we have
\begin{gather}
{\mathfrak m}_{k', B'} \circ \hat{\mathfrak m}_{k'',B''} (\xi_1, \dots, \xi_k) \nonumber\\
 \qquad{} =
\sum_{j=1}^k (-1)^{\sum_{i=1}^{j-1} \vert \xi_i \vert'} {\mathfrak m}_{k',B'} \bigl(\xi_1, \dots, {\mathfrak m}_{k'',B''}\bigl(\xi_j, \dots, \xi_{j+k''-1}\bigr), \dots, \xi_k\bigr)\nonumber\\
 \qquad{}= \sum_{j=1}^k (-1)^{\delta_4} \bigl({\rm ev}_0^{B'}\bigr)_! \bigl( {\rm ev}_1^{B'*}\xi_1 \wedge \dots \wedge {\rm ev}_{j-1}^{B'*} \xi_{j-1}\nonumber\\
 \qquad\hphantom{= \sum_{j=1}^k}{}
  \wedge {\rm ev}_j^{B'*} \bigl( \bigl({\rm ev}_0^{B''}\bigr)_! \bigl({\rm ev}_1^{B''*} \xi_j \wedge \dots \wedge {\rm ev}_{k''}^{B''*} \xi_{j+k''-1} \bigr)\bigr) \wedge \dots \wedge {\rm ev}_{k'}^{B'*} \xi_k \bigr),\label{equality1}
\end{gather}
where
\begin{equation*}
\delta_4 = \sum_{i=1}^{j-1} \vert \xi_i \vert' + \epsilon\bigl(\xi_1, \dots, {\mathfrak m}_{k'',B''}\bigl(\xi_j, \dots, \xi_{j+k''-1}\bigr), \dots, \xi_k\bigr) + \epsilon\bigl(\xi_j, \dots, \xi_{j+k''-1}\bigr),
\end{equation*}
and
\smash{${\rm ev}_j^{(B',B'')}\colon {\mathcal M}_{k'+1}(B';{\mathcal L}';{\mathcal R}')_{{\rm ev}_j^{B'}} \times_{{\rm ev}_0^{B''}} {\mathcal M}_{k''+1}(B'';{\mathcal L}'';{\mathcal R}'') \to R_{\alpha_j}$} is
the evaluation map at the $j$-th marked point on the fiber product
\begin{equation*}
{\mathcal M}_{k'+1}\bigl(B';{\mathcal L}';{\mathcal R}'\bigr)_{{\rm ev}_j^{B'}} \times_{{\rm ev}_0^{B''}} {\mathcal M}_{k''+1}\bigl(B'';{\mathcal L}'';{\mathcal R}''\bigr).
\end{equation*}
Here the numbering of the marked points is the same as that on ${\mathcal M}_{k+1}(B;{\mathcal L};{\mathcal R})$.
We also have
\begin{gather*}
  \bigl({\rm ev}_0^{B'}\bigr)_! \bigl( {\rm ev}_1^{B'*}\xi_1 \wedge \dots \wedge {\rm ev}_{j-1}^{B'*} \xi_{j-1}
 \wedge {\rm ev}_j^{B'*} \bigl( \bigl({\rm ev}_0^{B''}\bigr)_! \bigl({\rm ev}_1^{B''*} \xi_j \wedge \dots \wedge {\rm ev}_{k''}^{B''*} \xi_{j+k''-1} \bigr)\bigr)  \\ \hphantom{\bigl({\rm ev}_0^{B'}\bigr)_! \bigl(}{}
   \wedge {\rm ev}_{j+1}^{B'*} \xi_{j+k''} \wedge \dots \wedge {\rm ev}_{k'}^{B'*} \xi_k \bigr)  \\
\qquad{} =  (-1)^{\eta_1}
\bigl({\rm ev}_0^{B'}\bigr)_! \bigl( \bigl( {\rm ev}_1^{B'*}\xi_1 \wedge \dots \wedge {\rm ev}_{j-1}^{B'*} \xi_{j-1} \wedge {\rm ev}_{j+1}^{B'*} \xi_{j+k''} \wedge \dots \wedge {\rm ev}_{k'}^{B'*} \xi_k \bigr)  \\
\hphantom{=(-1)^{\eta_1}\bigl({\rm ev}_0^{B'}\bigr)_! \bigl(}\qquad{}
  \wedge {\rm ev}_j^{B'*} \circ \bigl({\rm ev}_0^{B''}\bigr)_! \bigl({\rm ev}_1^{B''*} \xi_j \wedge \dots \wedge {\rm ev}_{k''}^{B''*} \xi_{j+k''-1} \bigr)\bigr)   \\
\qquad{} =  (-1)^{\eta_1}
\bigl({\rm ev}_0^{B'}\bigr)_! \bigl( \bigl( {\rm ev}_1^{B'*}\xi_1 \wedge \dots \wedge {\rm ev}_{j-1}^{B'*} \xi_{j-1} \wedge {\rm ev}_{j+1}^{B'*} \xi_{j+k''} \wedge \dots \wedge {\rm ev}_{k'}^{B'*} \xi_k \bigr)  \\
\hphantom{=(-1)^{\eta_1}\bigl({\rm ev}_0^{B'}\bigr)_! \bigl(}\qquad{}
  \wedge (\pi_{B'})_! \circ \pi_{B''}^* \bigl({\rm ev}_1^{B''*} \xi_j \wedge \dots \wedge {\rm ev}_{k''}^{B''*} \xi_{j+k''-1} \bigr)\bigr)  \\
\qquad{} =  (-1)^{\eta_1}
\bigl({\rm ev}_0^{B'}\bigr)_! \circ (\pi_{B'})_! \bigl( \pi_{B'}^* \bigl({\rm ev}_1^{B'*}\xi_1 \wedge \dots \wedge {\rm ev}_{j-1}^{B'*} \xi_{j-1} \wedge {\rm ev}_{j+1}^{B'*} \xi_{j+k''} \wedge \dots \wedge {\rm ev}_{k'}^{B'*} \xi_k\bigr)  \\
\hphantom{=  (-1)^{\eta_1}\bigl({\rm ev}_0^{B'}\bigr)_! \circ (\pi_{B'})_! \bigl(}\qquad{}
  \wedge \pi_{B''}^* \bigl({\rm ev}_1^{B''*} \xi_j \wedge \dots \wedge {\rm ev}_{k''}^{B''*} \xi_{j+k''-1} \bigr) \bigr)  \\
\qquad{} =  (-1)^{\eta_1 + \eta_2}
\bigl({\rm ev}_0^{B'} \circ \pi_{B'}\bigr)_! \bigl( \pi_{B'}^* \bigl({\rm ev}_1^{B'*}\xi_1 \wedge \dots \wedge {\rm ev}_{j-1}^{B'*} \xi_{j-1} \bigr)
 \\
\hphantom{=  (-1)^{\eta_1 + \eta_2}\bigl({\rm ev}_0^{B'} \circ \pi_{B'}\bigr)_! \bigl(}\qquad{}
  \wedge \pi_{B''}^* \bigl({\rm ev}_1^{B''*} \xi_j \wedge \dots \wedge {\rm ev}_{k''}^{B''*} \xi_{j+k''-1} \bigr) \\
\hphantom{=  (-1)^{\eta_1 + \eta_2}\bigl({\rm ev}_0^{B'} \circ \pi_{B'}\bigr)_! \bigl(}\qquad{}
  \wedge \pi_{B'}^* \bigl( {\rm ev}_{j+1}^{B'*} \xi_{j+k''} \wedge \dots \wedge {\rm ev}_{k'}^{B'*} \xi_k \bigr)\bigr) \\
\qquad{} =  (-1)^{\eta_1 + \eta_2}
\bigl({\rm ev}_0^{(B',B'')}\bigr)_! \bigl({\rm ev}_1^{(B',B'') *} \xi_1 \wedge \dots \wedge {\rm ev}_k^{(B',B'') *} \xi_k\bigr),
\end{gather*}
where
\begin{gather*}
\eta_1= \Biggl( \Biggl(\sum_{i=j}^{j +k'' -1} \deg\xi_i\Biggr) + \Biggl(\mu(R_{\alpha}) - \sum_{i=j}^{j+k''-1} \mu\bigl(R_{\alpha_i}\bigr) + k'' -2\Biggr) \Biggr) \Biggl(\sum_{i=j+k''}^k \deg \xi_i\Biggr),
\\
\eta_2 = \Bigg(\sum_{i=j}^{j +k'' -1} \deg \xi_i\Bigg) \Bigg(\sum_{i=j+k''}^k \deg \xi_i\Bigg).
\end{gather*}
The second equality is a consequence of Proposition \ref{base change} (base change formula) for integration along fibers, i.e., ${\rm ev}_j^{B'*} \circ \bigl({\rm ev}_0^{B''}\bigr)_! = (\pi_{B'})_! \circ \pi_{B''}^*$.
The third equality follows from Corollary \ref{iteratefib}.
Note that
\begin{equation*}
{\rm ev}_i^{(B',B'')} =
\begin{cases}
{\rm ev}_i^{B'} \circ \pi^B_{B'} & i= 0, 1, \dots, j-1, \\
{\rm ev}_{i-j+1}^{B''} \circ \pi^B_{B''}, & i= j, \dots, j + k'' -1, \\
{\rm ev}_{i -k''+1}{B'} \circ \pi^B_{B'}, & i= j + k'', \dots, k.
\end{cases}
\end{equation*}
We set
\[\delta_5 = \eta_1 + \eta_2 = \Biggl( \mu(R_{\alpha}) - \sum_{i=j}^{j+k''-1} \mu\bigl(R_{\alpha_i}\bigr) + k'' -2\Biggr) \Biggl(\sum_{i=j+k''}^k \deg \xi_i\Biggr).\]
Then we have
\begin{gather} 
 \bigl({\rm ev}_0^{B'}\bigr)_! \bigl( {\rm ev}_1^{B'*}\xi_1 \wedge \dots \wedge {\rm ev}_{j-1}^{B'*} \xi_{j-1}
 \wedge {\rm ev}_j^{B'*} \bigl( \bigl({\rm ev}_0^{B''}\bigr)_! \bigl({\rm ev}_1^{B''*} \xi_j \wedge \dots \wedge {\rm ev}_{k''}^{B''*} \xi_{j+k''-1} \bigr)\bigr) \nonumber \\ \hphantom{\bigl({\rm ev}_0^{B'}\bigr)_! \bigl(}{}
 \wedge {\rm ev}_{j+1}^{B'*} \xi_{j+k''} \wedge \dots \wedge {\rm ev}_{k'}^{B'*} \xi_k \bigr) \nonumber \\
\qquad{} = (-1)^{\delta_5} \bigl({\rm ev}_0^{(B',B'')}\bigr)_! \bigl({\rm ev}_1^{(B',B'')*} \xi_1 \wedge \dots \wedge {\rm ev}_k^{(B',B'')*} \xi_k\bigr). \label{equality2}
\end{gather}
Set $\kappa' = \delta_4 + \delta_5$, i.e.,
\begin{gather*}
\kappa' = \sum_{i=1}^{j-1} \vert \xi_i \vert' + \epsilon\bigl(\xi_1, \dots, {\mathfrak m}_{k'',B''}\bigl(\xi_j, \dots, \xi_{j+k''-1}\bigr), \dots, \xi_k\bigr) + \epsilon\bigl(\xi_j, \dots, \xi_{j+k''-1}\bigr)  \\ \hphantom{\kappa' =}{}
 + \Biggl( \mu(R_{\alpha}) - \sum_{i=j}^{j+k''-1} \mu\bigl(R_{\alpha_i}\bigr) + k'' -2\Biggr) \Biggl(\sum_{i=j+k''}^k \deg \xi_i\Biggr).
\end{gather*}

Recall the definitions of the shifted degree in~\eqref{degree} and the $\epsilon(\xi_1, \dots, \xi_k)$ in~\eqref{epsilon} and the fact on
the degree of ${\mathfrak m}_k$ \eqref{deg m_k}, we find that
\begin{gather*}
\kappa'\equiv \epsilon(\xi_1, \dots, \xi_k) + \sum_{i=1}^{k} \deg \xi_i - k - 1 + j + k' \Biggl( \mu(R_{\alpha}) - \sum_{i=j}^{j+k''-1} \mu\bigl(R_{\alpha_i}\bigr) \Biggr)
 \\ \hphantom{\kappa'\equiv }{}
+ \Biggl(\sum_{p=1}^{j-1} \mu\bigl(R_{\alpha_p}\bigr)\Biggr)\Biggl(\mu(R_{\alpha}) - \sum_{p=j}^{j+k''-1} \mu\bigl(R_{\alpha_p}\bigr)\Biggr) + \bigl(k'-j\bigr)k'' \mod 2.
\end{gather*}

Combining~\eqref{equality1} and~\eqref{equality2}, we obtain Lemma \ref{lemma}.
\end{proof}

Now we show the following.
\begin{thm}\label{conclusion}
The operations ${\mathfrak m}_k$, $k = 0, 1, \dots$, that is the Bott--Morse $A_{\infty}$-operation in the de Rham model, in Definition $\ref{def}$ satisfy the filtered $A_{\infty}$-relation
\begin{equation*}
\sum_{k' + k'' = k+1} {\mathfrak m}_{k'} \circ \widehat{\mathfrak m}_{k''} (\xi_1, \dots, \xi_k) = 0 \qquad {\rm for} \ k=1,2, \dots.
\end{equation*}
\end{thm}

\begin{proof}
By Proposition \ref{bdryori}, we find the following.

\begin{clm}\label{rel3}
The summation of the right hand side of~\eqref{lem} over $k'$, $k''$, $B'$, $B''$ such that
${k' + k'' = k+1}$, $B=B' + B''$, $(k', B'), (k'', B'') \neq (1,0)$ is equal to
\[ (-1)^{\kappa + \kappa'} \bigl({\rm ev}^B_0\vert_{\partial {\mathcal M}_{k+1}(B;{\mathcal L}; {\mathcal R})}\bigr)_! \bigl({\rm ev}_1^{B*}\xi_1 \wedge \dots\wedge {\rm ev}_k^{B*} \xi_k\bigr).\]
\end{clm}

Note that
\begin{align*}
\kappa + \kappa' &{} \equiv \epsilon(\xi_1 ,\dots, \xi_k) + 1 + k + \sum_{i=1}^k \deg \xi_i + \dim R_{\alpha_0} + \mu(R_{\alpha_0}) - \sum_{p=1}^k
\mu\bigl(R_{\alpha_p}\bigr)  \\
&{} \equiv \epsilon(\xi_1 ,\dots, \xi_k) + 1 + \dim {\mathcal M}_{k+1}(B;{\mathcal L};{\mathcal R}) + \sum_{i=1}^k \deg \xi_i \mod 2 .
\end{align*}

Using Proposition \ref{corrStokes}, we have
\begin{gather}
{\rm d} \bigl({\rm ev}_0^B\bigr)_!\bigl({\rm ev}_1^{B*}\xi_1 \wedge \dots \wedge {\rm ev}_k^{B*}\xi_k\bigr)  =  \bigl({\rm ev}_0^B\bigr)_! {\rm d}\bigl({\rm ev}_1^{B*}\xi_1 \wedge \dots \wedge {\rm ev}_k^{B*} \xi_k\bigr) \nonumber \\
\qquad{} + (-1)^\nu \bigl({\rm ev}_0^B\vert_{\partial {\mathcal M}_{k+1}(B;{\mathcal L};{\mathcal R})}\bigr)_!
\bigl({\rm ev}_1^{B*} \xi_1 \wedge \dots \wedge {\rm ev}_k^{B*}\xi_k\bigr),\label{Stokes}
\end{gather}
where $\nu = \dim {\mathcal M}_{k+1}(B;{\mathcal L};{\mathcal R}) + \sum_{i=1}^k \deg \xi_i.$

Combining~\eqref{rel1},~\eqref{rel2}, Claim \ref{rel3} and~\eqref{Stokes}, we have
\begin{gather*}
 {\mathfrak m}_{1,0} \circ {\mathfrak m}_{k, B} (\xi_1, \dots, \xi_k) + {\mathfrak m}_{k,B} \circ \hat{\mathfrak m}_{1,0} (\xi_1, \dots, \xi_k)  \\
\qquad{} + \sum_{(k',B'), (k'', B'') \neq (1,0)} {\mathfrak m}_{k',B'} \circ \hat{\mathfrak m}_{k'',B''} (\xi_1, \dots, \xi_k) = 0
\end{gather*}
for all $(k,B) \neq (1,0)$.
Recall that, in the case that $(k, B) = (1,0)$, ${\mathfrak m}_{1,0}= {\rm d}$ clearly satisfies ${\mathfrak m}_{1,0} \circ {\mathfrak m}_{1,0} = 0$.
Hence, we obtain Theorem \ref{conclusion}.
\end{proof}

\subsection*{Acknowledgements}
The author is partially supported by JSPS Grant-in-Aid for Scientific Research 19H00636, 24H00182.
He is also grateful for National Center for Theoretical Sciences, Taiwan, where a~part of this work
was carried out.

\pdfbookmark[1]{References}{ref}
\LastPageEnding

\end{document}